\documentclass{amsart}
\usepackage{verbatim}

\usepackage{hyperref}
\usepackage{txfonts}

\usepackage{amssymb}
\usepackage[active]{srcltx}
\usepackage{amsmath, amssymb, euscript,  mathrsfs, enumerate, amsthm, amsfonts, graphicx, color}
\usepackage{calligra}   

\makeatletter
\newcommand{\labitem}[2]{%
\def\@itemlabel{\textbf{#1}}
\item
\def\@currentlabel{#1}\label{#2}}
\makeatother

\newcommand{\R}{{\mathbb R}}

\newcommand{\fr} {\frac}
\newcommand{\e}{\epsilon}

\newcommand{\ds}{\displaystyle}

\newtheorem{thm}{Theorem}[section]
\newtheorem{theorem}{Theorem}[section]
\newtheorem{lemma}[thm]{Lemma}
\newtheorem{cor}[thm]{Corollary}
\newtheorem{remark}[thm]{Remark}
\newtheorem{prop}[thm]{Proposition}
\newtheorem{definition}[thm]{Definition}

\newtheorem*{notation}{Notation}
\newtheorem{claim}[thm]{Claim}

\theoremstyle{definition}

\begin{document}
\title{Yamabe Flow: steady solitons and type II singularities}

\author[Beomjun Choi]{Beomjun Choi}
\address{Beomjun Choi :
Columbia University, Department of Mathematics, 2990 Broadway, New York, NY 10027, USA}
\email{bc2491@columbia.edu }

\author[Panagiota Daskalopoulos]{Panagiota Daskalopoulos}
\address{Panagiota Daskalopoulos: Columbia University, Department of Mathematics, 2990 Broadway, New York, NY 10027, USA}
\email{pdaskalo@math.columbia.edu}


\begin{abstract}
We study the convergence of  complete non-compact conformally flat  solutions to the Yamabe flow to Yamabe steady solitons. 
We also prove the existence of Type II singularities which develop  at either a finite time $T$ or as $t \to +\infty$. 

 \end{abstract}
 \maketitle


\section{Introduction}\label{sec-introduction}

Let $(M,g_0)$ be a Riemannian  manifold without boundary  of dimension $n \geq 3$. If $g = u^{\frac 4{n+2}} \, g_0$
is a metric conformal to $g_0$, the scalar curvature $R$  of $g$ is given in terms of the
scalar curvature $R_0$ of $g_0$ by
$$R= u^{-1} \, \big ( - \bar c_n \Delta_{g_0} u^{\frac{n-2}{n+2}} + R_0 \, u^{\frac {n-2}{n+2}} \big )$$
where $\Delta_{g_0}$ denotes the Laplace Beltrami operator with respect to $g_0$ and $\bar c_n =  4 (n-1)/(n-2)$.

In 1989 R. Hamilton introduced the  {\em Yamabe flow}
\begin{equation}
\label{eq-yamabe}
\frac{\partial g}{\partial t} = -R\, g
\end{equation}
as an approach to solve the {\em Yamabe problem}  on manifolds of positive conformal Yamabe invariant.

In the case where $M$ is compact, the long time existence and convergence of Yamabe flow is well understood. 
Hamilton \cite{H} himself  showed the existence of the  normalized Yamabe flow (which is the re-parametrization of (\ref{eq-yamabe}) to keep the volume fixed)  
 for all time; moreover, in the case when the scalar curvature of the initial metric is negative,  he showed   the exponential convergence  of the  flow  to a  metric of constant scalar curvature. 
Chow \cite{Ch} showed the  convergence of the flow,  under the conditions that the initial metric is locally conformally flat and of positive Ricci curvature.
The convergence of the flow for any locally conformally flat initially metric was shown by Ye  \cite{Y}.

Schwetlick and Struwe  \cite{SS}  obtained the convergence
of the  Yamabe flow on a general compact manifold under
a suitable Kazdan-Warner type of condition that rules out the formation of bubbles
and this condition is verified  (via the positive mass Theorem) in dimensions $3 \leq n \leq 5$.
 The convergence  result,   in its full generality, was established by
 Brendle  \cite{S2} and \cite{S1}   (up to a technical assumption, in dimensions $n \ge 6$,   on the rate of vanishing of Weyl tensor at the points at which it vanishes):
starting with any smooth  metric on a compact manifold, the normalized Yamabe flow   converges to a metric of constant scalar curvature.

Although   the  Yamabe flow on compact manifolds is well understood,  the complete non-compact case is  unsettled. Even though the analogue of Perelman's monotonicity formula is still lacking for the Yamabe flow, one expects that gradient Yamabe soliton solutions  model finite and infinite time singularities. These are special solutions $g=g_{ij}$ of the Yamabe flow 
\eqref{eq-yamabe} for which there exist a 
{\em potential  function $P(x,t)$}  so that
$$(R-\rho)g_{ij} = \nabla_i\nabla_j P, \qquad \rho \in \{1,-1,0\} $$
where the covariant derivatives on the right hand side are taken with respect to metric $g(\cdot,t)$. 
Depending on the sign of the constant $\rho$,  a Yamabe soliton is  called a {\em Yamabe shrinker}, a {\em Yamabe expander} 
or a { Yamabe steady soliton}  if $\rho=1,-1$ or $0$ respectively. 

The classification of locally conformally flat Yamabe solitons
with positive  sectional curvature was  established in \cite{DS} (see also \cite{CSZ} and \cite{CMM}). 
It is shown  in \cite{DS} that such solitons are globally conformally equivalent to $\R^n$ and   correspond to self-similar solutions of the fast-diffusion equation 
\begin{equation}\label{eq-flatyamabe}
u_t = \fr{n-1}{m} \Delta u^m, \qquad  m=\fr{n-2}{n+2}
\end{equation}
satisfied by the conformal factor defined by $ g_{ij}=u^{\fr{4}{n+2}} \delta_{ij} $.   A complete description of those solutions is  given in \cite{DS}. 
In \cite{CSZ} the assumption of positive sectional curvature was relaxed to that of nonnegative Ricci curvature.

The works   \cite{DKS, DS}  address the  {\em singularity formation}  of  complete non-compact solutions to the conformally flat Yamabe flow whose conformal factors 
have {\em cylindrical behavior at infinity}. It was shown  in these works  that the  singularity profiles of such solutions are Yamabe solitons
which are  determined  by the second order asymptotics at infinity of the initial data
which is  matched  with  that of  the  corresponding self-similar solution. 
The  solutions  may become  extinct at the extinction
time $T$ of the cylindrical tail  or may live longer than $T$. In  the first case,  the singularity profile  is described by a {\em Yamabe shrinker}  that becomes extinct at time $T$. 
This result can be seen as a stability result around the Yamabe shrinkers with cylindrical behavior at infinity. 
In the second  case, the flow develops a singularity at time $T$  which is described by a {\em singular}  Yamabe shrinker slightly before $T$
and by  a matching  {\em Yamabe expander}  slightly after $T$. All such singularities are of {\em type I}.

%

In this paper, we address singularities which are modeled  on {\em Yamabe steady solitons}.  In Theorem \ref{thm-1}, we  find a condition on 
a conformally flat initial data $g_0= u^{\fr{4}{n+2}}_0 \delta_{ij} $ under  which  
the Yamabe flow  converges, as $t \to +\infty$,  to a steady gradient soliton.  In   Theorem \ref{thm-2}, we 
we study a more general class  of  non-smooth initial data. 
In the Section 4,  we  provide conditions on  a complete non-compact and conformally flat initial data $g_0= u^{\fr{4}{n+2}}_0 \delta_{ij} $ which  guarantee 
that the Yamabe flow  will form a   type II  singularity. We show the existence of both {\em finite time} and {\em infinite time}
type II singularities in Theorems \ref{thm-typeII-1} and \ref{thm-typeII-2}, respectively. To our knowledge, this is the first time that a  type II singularity has been shown to exist in the Yamabe flow.  

In what follows we will simply say that a metric $g_{ij}$  is conformally flat if it is globally conformally flat over $\R^n$,
namely $g_{ij}=u^{\fr{4}{n+2}}\delta_{ij}$ for a conformal factor $u$ defined on $\R^n$, and we will often use the notation $(\R^n, g_{ij})$ to denote such 
a metric.

It was shown in \cite{DS} and \cite{H},  that if $g_{ij}=u^{\fr{4}{n+2}}\delta_{ij}$ is a  conformally flat Yamabe steady gradient soliton with positive sectional curvature,   then u  is a smooth entire and rotationally symmetric  solution of the  elliptic equation 
\begin{equation}\label{eq-u}
\fr{n-1}{m}\Delta u^m  + \beta x \cdot \nabla  u+ \gamma u=0, \qquad \mbox{on} \,\, \R^n
\end{equation} 	
for parameters 
 \begin{equation} \beta\ge0 \qquad  \text{ and } \qquad \gamma=\fr{2\beta}{1-m}.
\end{equation}
Moreover, it follows by the results in \cite{DS} that for  each $\beta>0$ and $\gamma=\fr{2\beta}{1-m}$, the equation \eqref{eq-u} admits one parameter family of rotationally symmetric solutions $(u_{\beta,\lambda})_{\lambda>0}$,  satisfying the asymptotic behavior 
$${\ds u_{\beta, \lambda} ^{1-m} \sim O\left( \fr{\ln |x|}{|x|^2}\right)}, \qquad \mbox{as} \,\, |x| \to +\infty$$ 
which are  uniquely determined by their  value at the origin, 
that is 
\begin{equation}\label{eqn-lambda}
u_{\beta,\lambda}(0)=\lambda. 
\end{equation} 
It should be noted that for a fixed $\beta >0$,  ${\ds u(t)=u_{\beta, e^{-\gamma t}}}$,  for  $t \in (-\infty,\infty)$ is a solution of the Yamabe flow 
and hence $\lambda$ is just a time dilation parameter. 
Moreover, all $(\R^n,u^{1-m}_{\beta,\lambda}\delta_{ij})$ are isometric to each other by conformal changes $x\to ax$, $a>0$.

Hsu, in \cite{H}  obtained the  first order decay rate at infinity  of a Yamabe steady soliton $u_{\beta,\lambda}$. Namely, it was shown that  
 \begin{equation}\label{eq-firstdecay}
\lim_{|x|\to\infty} \fr{|x|^2 {u_{\beta,\lambda}}^{1-m}}{\ln|x|} = \fr{(n-1)(n-2)}{\beta}.
\end{equation}
In order to study the stability around these solitons, it is necessary to establish their  second order asymptotics at infinity. 
In Section \ref{sec-lowerorder}, we establish such asymptotics  showing that  second order decay depends on the parameter $\lambda$.

In Section \ref{sec-stability}, we  prove that if an initial conformally flat metric is asymptotically close to a  steady soliton $u_{\beta,\lambda}$ up to the second order, namely  $$u_0^{1-m} = \fr{1}{r^2} \left(  \fr{(n-2)(n-1)}{\beta} \ln r + K+ o(1) \right)$$ for some $K \in \R$, then  the rescaled solution $\bar u(x,t)=e^{\gamma t}u(e^{\beta t}x,t)$ converges, as $t \to +\infty$,  to  $u_{\beta,\lambda}$. The constant $\lambda$ is determined by $K$
through the second order decay rate  at infinity  of $u_{\beta, \lambda}$, namely ${\ds K=\fr{2\ln \lambda}{n+2}+\fr{\ln \beta}{2}+ \kappa(n)}$, for some universal constant $\kappa=\kappa(n)$.


Finally,  in Section \ref{sec-type2}  we construct examples of complete noncompact and globally conformally flat solutions 
of the  Yamabe flow which develop  type II singularities.  It has been  observed in \cite{DKS} that a  conformally flat Yamabe gradient shrinker $v_{\beta,\lambda}$ which vanishes at time $T$,
satisfies the asymptotic behavior  $$|x|^2\, v_{\lambda}^{1-m}(x)=(n-1)(n-2) \, T -B_\lambda \, |x|^{-\gamma}+o (|x|^{-\gamma}),
\qquad \mbox{as}\,\, |x| \to +\infty.$$ The key point is that the decay rate $\gamma>0 $ depends only on $\beta>0$, which is related with the scalar curvature at the tip (where the maximum scalar curvature occurs)  and this $\gamma:=\gamma(\beta) \to 0$,  as $\beta \to \infty$.  Thus, one may  guess that a solution may  develop a type II singularity,  if  its initial data  has slower second order  decay rate than any Yamabe shrinkers. We will therefore choose, for any given $T>0$,  an initial data 
$g_0 = u^\fr{4}{n+2}_0 \delta_{ij}$ such that the tail of $|x|^2 \, u_0^{1-m}(x)-(n-1)(n-2)\, T$ decays slower than any  power $|x|^{-\gamma}$,  with $\gamma>0$, and prove that the solution with initial data $g_0$ will develop a type II singularity at its   extinction time $T$.  This idea is similar to  that in \cite{HR} where Hamel and Roques found an accelerating fast front propagation for the  KPP type equation $u_t= u_{xx} + f(u)$ and  for slowly decaying initial data. 
We will also find a class of initial data for which the Yamabe flow  develops a type II singularity,  as $t \to +\infty$.  

\section{Lower Order Asymptotics}\label{sec-lowerorder}
In this section, we will derive the  second and third order asymptotics of conformally flat radial steady gradient solitons $u_{\beta,\lambda}$, as $r=|x| \to +\infty$. As we saw in the introduction these are solutions of the elliptic equation 
\begin{equation}\label{eq-u2}
\fr{n-1}{m}\Delta u^{m}  + \beta x \cdot \nabla  u+ \gamma u=0, \qquad \mbox{on} \,\, \R^n, \qquad m=\frac{n-2}{n+2}
\end{equation} 	
with parameters 
$$\beta\ge 0 \qquad \mbox{and} \qquad     \gamma=\fr{2\beta}{1-m}$$
and for each $\beta >0$ they are uniquely determined by  their value at the origin $\lambda:=u_{\beta,\lambda}(0)$.  

For the remaining of the section we fix $\beta >0$ and $\lambda >0$ and set for simplicity   $u(r) :=u_{\beta,\lambda}(x)$, $r=|x|$.  It is 
convenient to work in cylindrical coordinates $s=\log r$. Using this change,  
the radial metric $g=u(r)^{1-m}\, dx^2$ is expressed as $$ g=u(r)^{1-m}(dr^2+r^2g_{S^{n-1}}) = w(s) (ds^2 +g_{S^{n-1}})$$ where 
$$w(s)=r^2 u(r)^{1-m}.$$ 

Using that $m=\frac{n-2}{n+2}$, we find by direct calculation that  \eqref{eq-u2}   translates into  the following equation for $w$
\begin{equation}\label{eq-w}
w_{ss}=\fr{6-n}4  \cdot \fr{w_s^2}{w}+\left(n-2-\fr{\beta}{n-1} w_s  \right)\, w, \qquad \text{for } \,\, s\in(-\infty,\infty).
\end{equation}

We recall in the next Proposition  previous  results regarding the first order asymptotics of $w$ which were shown in \cite{DS} and \cite{H}. 
4
\begin{prop}[\cite{DS}, \cite{H}] \label{prop-prelim} For a conformally flat and radially symmetric steady gradient soliton $w$, we have  $w>0$ and $w_s>0$ for all $s\in\R$. Moreover, there are positive constants $C_1$ and $C_2$ such that $C_1\le w_s \le C_2$,  for $s\ge0$ and 
\begin{equation}\label{eqn-ws}
w_s \to \fr{(n-1)(n-2)}{\beta}, \qquad \text{ as } \, s\to \infty.
\end{equation}
\end{prop}
To find the second order asymptotics of $w$ we set  
$$w:=\fr{(n-1)(n-2)}{\beta}s +h.$$
Plugging this into \eqref{eq-w}, we obtain the following equation for $h$ 
\begin{equation}\label{eqn-h1} 
L h:=h_{ss}+ (n-2) \, s\, h_s=\fr{6-n}4 \cdot \fr{w_s^2}{w}- \frac{\beta}{n-1} \,  h\, h_s=:f(s).
\end{equation}

The signs of  $h_{ss}=w_{ss}$ and ${\ds h_s=w_s-\fr{(n-1)(n-2)}{\beta}}$ can be determined,  depending on dimension $n$ as shown next. 

\begin{prop}\label{prop-wss}
For $n\ge6$, $h_{ss}=w_{ss}>0$ and $h_s<0$ for all $s\in (-\infty,\infty)$. For $3\le n<6$, there exist $s_0>0$ such that $h_{ss}=w_{ss}<0$ and $h_s>0$  for all $s>s_0$.
\end{prop} 
\begin{proof}
Recall that  $w=r^2u^{1-m}>0$. Differentiating  in $s$ gives that $w_s=rw_r=2r^2u^{1-m}+(1-m)r^3u^{-m}u_r$. Using that   
$\lim_{r\to0}u(r)>0$ and $\lim_{r\to0}u_r(r)=0$, we conclude $\lim_{s \to -\infty} w_s \to 0$ and moreover we can check
\begin{equation}\label{eqn-wss}
\fr{w_{ss}}{w}= \fr{6-n}{4}\, \frac{w_s^2}{w^2} +(n-2) -\fr{\beta}{n-1}w_s \rightarrow 4, \quad  \text{ as } \,\, s\to -\infty.
\end{equation}
Hence, for any dimension $n$, we have $w_{ss} >0$ near $s=-\infty$.  Assume now that  $n >6$.  We will show that $w_{ss} >0$ is preserved  for all $s \in \R$. 
To this end, we   differentiate  \eqref{eq-w} in $s$ and obtain   
\begin{equation}\label{eq-w_sss}
(w_{ss})_s=\fr{6-n}4 \, \left( \fr{2w_s w_{ss}}{w}-\fr{w_s^3}{w^2}\right)+\left(n-2-\fr{\beta}{n-1} w_s  \right)w_s-\fr{\beta}{n-1}w_{ss}w.
\end{equation} 
Suppose $w_{ss}=0$ at some $s=\bar s$.  From  \eqref{eq-w}, using $6-n<0$, $w>0$ and $w_s>0$, we get 
$$0<w_s(\bar s) <\fr{(n-2)(n-1)}{\beta}.$$  Plugging this bound into \eqref{eq-w_sss}, gives   $(w_{ss})_s(\bar s) >0$. 
Hence,  $w_{ss}>0$ for all $s\in \R$, implying that $h_{ss}=w_{ss} >0$. Since, from Proposition \ref{prop-prelim} we have 
$$h_s = w_s - \fr{(n-2)(n-1)}{\beta} \to 0, \qquad \mbox{as} \,\,\, s\to \infty$$
 we also conclude that $h_s<0$. 

When $n=6$, equation \eqref{eq-w} could be viewed as a 1st order linear equation of $w_s$ assuming $w$ is given. Hence we can 
integrate from $s=-\infty$ and use  $\lim_{s \to -\infty} w_s=0$ to obtain 
 $$w_s(s)=\fr{(n-1)(n-2)}{\beta} \big ( 1- e^{  -\fr{\beta}{n-1}\int_{-\infty}^sw(l) \, dl } \,  \big )$$
 from which the bounds 
$$w_{ss}>0 \qquad \mbox{and} \qquad w_s <\fr{(n-1)(n-2)}{\beta}, \quad \mbox{for all }\,\, s\in R$$
readily follow. We conclude in this case that  $h_{ss}=w_{ss} >0$ and $h_s <0$ for all $s \in \R$. 

Finally,  assume that  $3<n<6$. By \eqref{eqn-wss} we have $\lim_{s\to -\infty} {w_{ss}}/{w}=4>0$.  If we have $w_{ss}(s_0) =0$ 
at some $s=s_0$, then  by \eqref{eq-w} and because $3<n<6$ at this time, we have 
${\displaystyle w_s(s_0) > \fr{(n-2)(n-1)}{\beta}}$
and hence $(w_{ss})_s(s_0) <0 $ by \eqref{eq-w_sss}. 
Thus,  $w_{ss}(s)<0$, for all $s >s_0$. 
On the contrary, if  we don't have such a point $s_0$, then  $h_{ss}=w_{ss}>0$ for all $s\in \R$. 
Thus  \eqref{eqn-ws}  implies that ${\ds w_s < \fr{(n-2)(n-1)}{\beta}}$ and hence 
${\ds w_{ss} > \fr{6-n}4 \fr{w_s^2}{w}}$ from equation \eqref{eq-w}. But since  ${\ds w_s \nearrow \fr{(n-1)(n-2)}{\beta}}$, as $ s \to +\infty$,
 there is some $c>0$ $$w_{ss} > \fr{6-n}4 \,  \fr{w_s^2}{w}>\fr{c}{s}, \quad \mbox{for} \,\, s >>1$$ 
implying that  $w_s \to \infty$,   a contradiction. We conclude that $w_{ss} <0$, for $s \geq s_0$,
for some $s_0 \in \R$, implying that $h_{ss} <0$ for all $s \geq s_0$.  Since, $w_{ss} <0$
and \eqref{eqn-ws} holds, we must also have $h_s >0$ for all $s > s_0$. 
This finishes the proof of the proposition. 
\end{proof}

Employing  the  previous Proposition, we can now  prove the following.

\begin{prop}\label{prop-hs}
For all $n\ge3$, we have 
\begin{equation}\label{eqn-prop-hs}
\lim_{s\to \infty}s^2 \, h_s=\fr{(6-n)(n-1)}{4\beta}.
\end{equation}
\begin{proof}
Begin by observing that by Proposition \ref{prop-prelim}, we have $h_s\to 0$ as $s\to \infty$, implying that  $|h|=o(s)$.
For $n\ge6$, we showed in Proposition \ref{prop-wss} that $h_s<0$. Therefore, it  follows from  \eqref{eqn-h1}, that for all $\epsilon>0$, there exists $ s_0>0$ such that for $s\ge s_0$ 
\begin{equation}\label{eq-hss}
L_\e h: = h_{ss}+\left((n-2)-\epsilon \right) sh_s = -\epsilon \, sh_s- \fr{\beta}{n-1}h h_s + \fr{6-n}4\fr{w_s^2}{w}\ge \fr{6-n}4\fr{w_s^2}{w}.\end{equation}
Multiplying  by $\exp \left ({\fr{s^2}{2} \left((n-2)-\epsilon\right)}\right )$ and integrating  from $s_0$ to $s$, we get 
\begin{equation}\label{eq-hs}\left  [h_s(l) \exp \left  ({\fr{l^2}{2} \big ((n-2)-\epsilon\big )}\right )\right]^s_{l=s_0}\ge \int_{s_0}^s \fr{6-n}4\fr{w_s^2}{w}\exp\left ({\fr{l^2}{2} \left((n-2)-\epsilon\right)}\right ) dl.\end{equation}
Setting  $I(s):=s^2 \exp \left ( {-\fr{s^2}{2} \big ((n-2)-\epsilon\big )}\right)$ and taking $\liminf_{s\to\infty}$, on  the LHS of \eqref{eq-hss}, gives 
$$ \liminf_{s\to\infty} \, [\mbox{I(s)} \,  \text{LHS}\eqref{eq-hs}]=\liminf_{s\to\infty} s^2 h_s.$$
For the RHS of \eqref{eq-hss}, we can apply   L'H\^{o}pital's  rule  to obtain
$$\begin{aligned}\lim_{s\to\infty} \, [\mbox{I(s)}\,  \text{RHS} \eqref{eq-hs}]&=\lim_{s\to\infty} \fr{\fr{6-n}4\fr{w_s^2}{w}\exp\left ({\fr{s^2}{2} \left((n-2)-\epsilon\right)}\right) }{\fr{\partial}{\partial s} \left [ s^{-2}\exp\left ({ \fr{s^2}{2} \left((n-2)-\epsilon\right)}\right ) \right ] }\\&= \fr{(6-n)(n-2)(n-1)}{4\beta}\big((n-2)-\epsilon\big)^{-1}. \end{aligned}$$
In the last equality, we used that ${\ds w_s \to \fr{(n-1)(n-2)}{\beta}}$  and ${\ds \frac{w}s \to \fr{(n-1)(n-2)}{\beta}} $ as $s\to \infty$.
Combining both sides, gives 
$$\liminf_{s\to\infty} s^2 h_s \ge \fr{(6-n)(n-2)(n-1)}{4\beta} \, \big((n-2)-\epsilon\big)^{-1} . $$
By taking $\epsilon\downarrow0$, we obtain  ${\ds \liminf_{s\to\infty} s^2 h_s \ge\fr{(6-n)\, (n-1)}{4\beta}}$. If we
chose $\epsilon<0$ in the beginning,  then we get the reversed inequality in \eqref{eq-hss} and  the same argument, 
yields  ${\ds \limsup_{s\to\infty} s^2 h_s \le \fr{(6-n)\, (n-1)}{4\beta}}$. We conclude that  \eqref{eqn-prop-hs} holds. 
For the remaining cases $3\le n <6$, once we choose $s_0$ so that $h_s>0$ on $[s_0,\infty)$, again a similar argument leads to the same conclusion. \end{proof}
\end{prop}

\begin{cor}\label{cor-asymp}
There exists a  constant $K=K(n,\beta,\lambda)$ such that 
 ${\ds h(s) = K+ \fr{(n-6)(n-1)}{4\beta} \fr{1}{s}+ o\left(\fr{1}{s}\right)}$, as $s \to +\infty$. It follows that     
$$w(s)=\fr{(n-2)(n-1)}{\beta}s + K + \fr{(n-6)(n-1)}{4\beta} \fr{1}{s} +  o\left(\fr{1}{s}\right) .$$
\begin{proof}
The convergence of $h(s)\to K(n,\beta,\lambda)$  readily follows  from the result in Proposition \ref{prop-hs}.
Define  ${\ds \bar h (s):=h(s)-K-\fr{(n-6)(n-1)}{4\beta}\fr{1}{s}}$ and integrate using again Proposition \ref{prop-hs} we obtain
${\ds | \bar h(s_1)| \le \int_{ s_1}^{\infty} |\bar h_s|\, ds}. $ Hence,  $\lim_{s\to\infty} s^2 \bar h_s =0$ 
yields  $\bar h(s)=o(s^{-1})$.
\end{proof}
\end{cor}

\begin{remark} In the special case  $n=6$, it is easy to see that  $|h_s|$ decays exponentially as $s \to +\infty$. 
Indeed,  by  \eqref{eq-hs}  and $h_s <0$,
it follows that  for each small $\epsilon >0 $, there exist two constant $c,C>0$ such that  
$$c \, \exp\big({-\fr{s^2}{2} \left(4+\epsilon\right)}\big ) \le |h_s| \le C \, \exp\left ({-\fr{s^2}{2} \left(4-\epsilon\right)}\right )$$ for large $s\ge s_0$.
\end{remark}

We will next  use the rich scaling properties of our equation \eqref{eq-u2} to  determine the value of the constant 
$K=K(n,\beta,\lambda)$ in Corollary \ref{cor-asymp}, up to an additive constant that  depends  only on 
the dimension $n$ and obtain the main result in this section which describes the asymptotic behavior for any steady soliton $u_{\beta,\lambda}$ up to third order. 
\begin{prop}\label{prop-asymp}
For $\beta >0, \lambda >0$, let  $u_{\beta,\lambda}$ denote the unique radially symmetric solution of equation \eqref{eq-u2} with $u_{\beta,\lambda}(0)=\lambda$. 
Then there exists  a constant  $\kappa=\kappa(n)\in \R$ depending only on dimension $n$  such that 
\begin{equation}\label{eqn-kappa} 
u_{\beta,\lambda} ^{1-m} (r)= \fr{(n-1)(n-2)}{\beta r^2} \left \{  \ln r +\left(\fr{2\ln \lambda}{n+2}+\fr{\ln \beta}{2}+ \kappa(n)  \right)+ \fr{(n-6)}{4(n-2)}\fr{1}{\ln r} +  o\left(\fr{1}{\ln r}\right)  \right \}.
\end{equation} 
\begin{proof}
For a radial solution $u$ of \eqref{eq-u}, the rescaling  $\tilde u(x) = A u(Bx)$ with $A,B>0$ becomes again radial solution of \eqref{eq-u} 
with the same $\beta$ and $\gamma$, if and only if $B=A^{\fr{1-m}{2}}$. These solutions are uniquely determined by their value at the origin. 
Hence, we have  \begin{equation}u_{\beta,\lambda_1}(r)=\fr{\lambda_1}{\lambda_2} u_{\beta, \lambda_2}\left(r\left(\fr{\lambda_1}{\lambda_2}\right)^{\fr{1-m}{2}} \right).\end{equation}
Similarly, by plugging into the equation \eqref{eq-u2}, the uniqueness again implies that 
\begin{equation} u_{\beta_1,\ \left(\fr{\beta_2}{\beta_1} \right)^{\fr{1}{1-m}}\lambda_2}(r)= \left(\fr{\beta_2}{\beta_1} \right)^{\fr{1}{1-m}}u_{\beta_2,\lambda_2} (r).\end{equation} Combining the two scalings above, gives 
\begin{equation}\label{eq-scaling}    u_{\beta,\lambda}(x)=(\lambda\beta^{\fr{1}{1-m}})\, 
u_{\beta,\beta^{-\fr{1}{1-m}}}((\lambda\beta^{\fr{1}{1-m}})^{\fr{1-m}{2}}x )= \lambda \, u_{1,1}((\lambda \beta^{\fr{1}{1-m}})^{\fr{1-m}{2}}x ).\end{equation}
By Corollary \ref{cor-asymp}, there is some $\kappa=\kappa(n)\in\R$ for which 
 $$u_{1,1} ^{1-m}(r) = \fr{(n-1)(n-2)}{r^2} \left(\ln r + \kappa(n) + \fr{n-6}{4(n-2)}\fr{1}{\ln r} +  o\left(\fr{1}{\ln r}\right) \right).$$ Direct computation using \eqref{eq-scaling} implies that \eqref{eqn-kappa} holds. 
\end{proof}
\end{prop}

\section{Long Time Stability}\label{sec-stability}

In the previous section, we found the asymptotic behavior at infinity, up to third order,  of any translating soliton  $u_{\beta,\lambda}$.
 It follows from  \eqref{eqn-kappa},  that the asymptotic behavior  up to second order is  sufficient 
 to distinguish among different steady solitons $u_{\beta,\lambda}$. 
Thus, it  is expected that for  an initial conformally flat metric $g_0 = u_0^{1-m} \, dx^2$ 
with  a behavior 
\begin{equation}\label{eqn-u0}u_0^{1-m}(x)=\fr{1}{|x|^2}\left (A \ln |x| + B + o(1)\right) \quad \text{as} \,\, |x| \to +\infty,
\quad \text{ for some } A>0,\ B\in\R
\end{equation} 
the  Yamabe flow $g(t)=u(\cdot,t)^{1-m} \, \delta_{ij}$ with initial data $g_0$ would converge, as $t \to +\infty$, and after 
rescaling  to the {\em unique} steady soliton $g_{\beta,\lambda}:= u_{\beta, \lambda}^{1-m}\delta_{ij}$   
having the same asymptotics of \eqref{eqn-u0}. 

\smallskip
 
In what follows we will show that this is indeed true. This will be done in two steps:  
In Theorem \ref{thm-1} we will establish  the $L^1_{loc}$ convergence of the flow,
 under the assumption that $u_0\in L^1_{loc}$ and satisfies \eqref{eqn-u0}. In Theorem \ref{thm-2},  we will provide  
an extra condition on $u_0$, namely that $u_0$ belongs to  the local Marcinkiewicz space $M^{(1-m)n/2}_{loc}$,  which guarantees the smooth convergence of the rescaled metric. While smooth globally conformally flat metrics are included in this space, it also allows certain singularities and degeneracies in the metric. In particular, certain cylindrical ends can be added at those singularity points and the flow starting  with this locally conformally flat metric also converges to a steady gradient soliton after those ends pinch  off in a finite time.

%

For a solution $u$ of \eqref{eq-flatyamabe} we consider the rescaled solution 
\begin{equation}\label{eqn-rescaling} 
\bar u(x,t):=e^{\gamma t} u(e^{\beta t}x,t), \qquad \beta >0, \,\, \gamma= \frac{2\beta}{1-m}.
\end{equation}
A direct calculation shows that $\bar u$ satisfies the equation 
\begin{equation}\label{eqn-u3}
\bar u_t = \fr{n-1}{m}\Delta \bar u^{m}  + \beta x \cdot \nabla  \bar u+ \gamma \bar u, \qquad \mbox{on} \,\, \R^n, 
\qquad m=\frac{n-2}{n+2}.
\end{equation} 	

The following result holds.

\begin{theorem}\label{thm-1}
Assume that   $g=u^{1-m} \, \delta_{ij}$ is  a  solution of the Yamabe flow  \eqref{eq-yamabe} with nonnegative initial data   
$u_0\in L^1_{loc}(\R^n)$  which has  the decomposition $u_0=\phi + \psi$ with  $\psi \in L^1(\R^n)$ and  
\begin{equation}\label{eqn-data}
\phi^{1-m} = \fr{1}{r^2} 
 \fr{(n-2)(n-1)}{\beta} \big ( \ln r + K + o(1) \big )\text{ \quad for some $\beta>0$ and $K\in \R$.}
 \end{equation}Then,   the rescaled solution $\bar u(x,t):=e^{\gamma t} u(e^{\beta t}x,t)$ converges, as $t \to +\infty$,  to $u_{\beta,\lambda}$ in $L^1_{loc}(\R^n)$,  for some $\lambda >0$. 
Moreover, the number $\lambda$ is uniquely determined by the coefficient $K$ in the asymptotic behavior of $u_0$, namely 
${\ds K=\fr{2\ln \lambda}{n+2}+\fr{\ln \beta}{2}+ \kappa(n)}$, for some universal constant $\kappa=\kappa(n)$. 
\end{theorem}

The proof of  Theorem \ref{thm-1} will be based on  the following   $L^1$-contraction property between two rescaled solutions 
$\bar u_1$ and $\bar u_2$ of equation \eqref{eq-flatyamabe}. 

\begin{lemma}\label{lemma-rescaledL1} If $u_1$ and $u_2$ are solutions of equation \eqref{eq-flatyamabe} and $\bar u_1$ and $\bar u_2$ are the rescaled solutions, respectively, then 
\begin{equation}\label{eqn-contrac1} 
\int_{\R^n} |\bar u_1(x,t) - \bar u_2(x,t) | \, dx \le e^{(\gamma-n\beta)t} \int_{\R^n} |\bar u_1(x,0) - \bar u_2(x,0) | \, dx. 
\end{equation} Note that  $ \gamma-n\beta = (\fr{2}{1-m}-n )\, \beta = \fr{2-n}{2}\beta <0$,   for $n\ge3$. 
\end{lemma}

\begin{proof} It is well known that any two solutions $u_1$ and $u_2$ of \eqref{eq-flatyamabe} satisfy the contraction principle
$$ 
\int_{\R^n} |u_1(x,t) -  u_2(x,t) | \, dx \le \int_{\R^n} | u_1(x,0) -u_2(x,0) | \, dx. 
$$
Hence, \eqref{eqn-contrac1} follows by direct calculation. 
\end{proof}

\begin{proof}[Proof of Theorem \ref{thm-1}]
By Lemma \ref{lemma-rescaledL1}, it suffices to prove the result when $u_0=\phi$. Consider the self-similar solution 
$u_{\beta,\lambda}$ satisfying the asymptotic behavior \eqref{eqn-kappa} with $\lambda$ determined by ${\ds K=\fr{2\ln \lambda}{n+2}+\fr{\ln \beta}{2}+ \kappa(n)}$.  It follows from  \eqref{eqn-kappa} and the  given asymptotics of initial data \eqref{eqn-data},
that for each $\epsilon>0$, there exists  $R_\epsilon >>1$ such that 
$$u_{\beta,\lambda-\epsilon}\le u_0 \le u_{\beta,\lambda+\epsilon}, \qquad \mbox{for} \,\, |x| \geq R_\epsilon.$$ 
Hence, we have  
$$\min (u_0, u_{\beta,\lambda-\epsilon}) -u_{\beta,\lambda-\epsilon} \in L^1(\R^n) \qquad \mbox{and} \qquad 
\max (u_0, u_{\beta,\lambda +\epsilon}) -u_{\beta,\lambda +\epsilon} \in L^1(\R^n).$$ 
Let $w, h$ denote the   solutions to equation \eqref{eq-flatyamabe} with initial data 
$\min (u_0, u_{\beta,\lambda-\epsilon}), \max (u_0, u_{\beta,\lambda+\epsilon}) $ respectively,
and denote by $\bar w, \bar v$ the rescaled solutions defined by  \eqref{eqn-rescaling}. 
The comparison principle then implies the inequality  
$$\bar w\le \bar u \le \bar h, \qquad \mbox{for}\,\, t>0$$ and by  Lemma \ref{lemma-rescaledL1}. 
we have 
$$\bar w \to u_{\beta,\lambda-\epsilon} \quad \mbox{and} \quad \bar h\to u_{\beta,\lambda+\epsilon}  \quad \mbox{in} \,\, L^1(\R^n),
\quad \mbox{as}\,\, t\to \infty.$$
For any compact set $K\subset \R^n$, we have $$\int_K | \bar u-u_{\beta,\lambda}|_+ \le \int_K |\bar h-u_{\beta,\lambda}|_+ \le \int_K |\bar h-u_{\beta,\lambda+\epsilon}|_+  + \int_K |\bar u_{\beta,\lambda+\epsilon}-u_{\beta,\lambda}|_+ .$$ Doing  the same computation for $\int_K |u-u_{\beta,\lambda}|_-$ 
and taking  $\limsup_{t\to\infty}$ yields 
$$\limsup_{t\to\infty} \int_K |u-u_{\beta,\lambda}| \le  \int_K |u_{\beta,\lambda+\epsilon}-u_{\beta,\lambda}|_+ +  \int_K |u_{\beta,\lambda-\epsilon}-u_{\beta,\lambda}|_- .$$
Taking $\epsilon \to 0$, the right hand side of above equality converges to $0$ and this finishes the proof.
\end{proof}

\smallskip

By the Arzela-Ascoli theorem, the $L^1_{loc}$ convergence in the previous result can be directly improved to $C_{loc}$ convergence,
when $\bar u(t)$ is locally equicontinuous for large $t$. It is well known, that for solutions $\bar u$  of  \eqref{eqn-u3}, an 
$L^\infty_{loc}$ bound implies equicontinuity (see  in Section 1.5 in \cite{DK}).  Thus, if we knew for instance that 
 \begin{equation}\label{eqn-toogood}
u_0  \leq u_{\beta,\lambda_0}, \qquad \mbox{for some} \,\, \lambda_0 >0
 \end{equation}
then we would know that $\bar u(t) \leq u_{\beta,\lambda_0}$ for all $t >0$ and 
as a consequence  $\bar u(t)$ would converge to  $u_{\beta,\lambda}$,
 as $t \to +\infty$, in $C_{loc}$.  Then,  standard regularity theory for uniformly parabolic equations would imply  $C_{loc}^\infty$ convergence. Condition \eqref{eqn-toogood} certainly  holds if $u_0 \in L^\infty$. Thus the following follows from our discussion above.

\begin{cor}\label{cor-uniform}
If $u_0\in L^\infty$ and $u_0^{1-m} = \fr{1}{r^2} \left(  \fr{(n-2)(n-1)}{\beta} \ln r + C+ o(1) \right)$, there is  some $\lambda_0>0$ such that $u_0 \le u_{\beta,\lambda_0}$.  
\begin{proof} For fixed $R>0$, since $u_{\beta,\lambda}$ is decreasing in $|x|$,  
$$\inf_{|x|\le R} u_{\beta,\lambda}^{1-m}(|x|)=u_{\beta,\lambda}^{1-m}(R)=\lambda^{1-m} u_{\beta,1}^{1-m}(\lambda^{\fr{1-m}{2}}R)=\fr{(n-2)(n-1)}{R^2\beta}\ln \lambda^{\fr{1-m}{2}}+ O(1) \to \infty$$ as $\lambda\to\infty$. 
i.e. $u_{\beta,\lambda}$ blow up on every compact sets as $\lambda \to \infty$. $L^\infty$ bound and decay asymptotics of initial data $u_0$ imply existence of a large $\lambda_0$ with $u_0 \le u_{\beta,\lambda_0}$.
\end{proof}
\end{cor}

Condition \eqref{eqn-toogood} is too restrictive and in particular does not allow any singularities or degeneracies 
in our initial metric.  The  object in the rest of this section is to give a condition on initial data which would guarantee  that  for some $t_0$ large we have 
 \begin{equation}\label{eq-t0}
 \bar u(\cdot, t_0) \leq u_{\beta,\lambda_0}, \qquad \mbox{for some} \,\, \lambda_0 >0 \quad \mbox{and} \quad t_0 >>1
 \end{equation}
and hence imply smooth convergence on compact sets. 

%
%

Next, we will show that  \eqref{eq-t0} holds for a  certain class of  locally conformally flat and possibly singular initial data.   The extra condition we will  assume is that $u_0$ belongs to the  Marcinkiewicz space  $ M^{p^*}_{loc}$, with $p^*=(1-m)\fr{n}{2}=\fr{4n}{2(n+2)}$.

To establish  that  \eqref{eq-t0} holds, we need an estimate which shows that a solution with non smooth, singular initial data becomes bounded and smooth. Such  smoothing estimates of the fast diffusion equation are well studied. 
If $u_0\in L^q_{loc}$,  with $q>(1-m)\fr{n}{2}=p^*$, then we have that  $u_0\in L^q $ due to our asymptotics and the
De Giorgi-Nash-Moser technique argument gives an $L^\infty$ estimate of $u$ for $t>0$.  In this critical exponent $q=p*$, however, this technique doesn't work and there is a surprising effect, the  so called delayed regularity phenomenon, which says that if $u_0$
belongs to the  Marcinkiewicz space  $ M^{p^*}_{loc}$, with $p^*=(1-m)\fr{n}{2}=\fr{4n}{2(n+2)}$, then  $u(t)$ eventually becomes  in $L^\infty$ for some $t_0 >0$ but it takes some time to get there.  For the convenience of the reader, we next define 
the space   $M^{p^*}_{loc}$ referring to  Chapter 1 and 6 of \cite{V} for further related preliminaries and details. 
		
		\begin{definition}[Marcinkiewicz Space] For an open set $\Omega\subset \R^n$\begin{equation}M^p(\Omega):= \large \{f \in L^1_{loc}(\Omega) | \,\, \exists C \text{ s.t. } \int_K|f| \, dx\le C \, |K|^{(p-1)/p}\text{ for all }|K|<\infty  \large \}\end{equation}		\begin{equation} ||f||_{M^p(\Omega)}= \sup \, \large \{|K|^{-(1-p)/p}\int_K |f| \, dx : K\subset \Omega, 0<|K|<\infty \large \}\end{equation} 
		\begin{equation}M^p_{loc}(\R^n) := \large \{f \in L^1_{loc}(\R^n) | f\in M^p(\Omega)\text{ for every bounded open set }\Omega
		\large \} \end{equation}\end{definition}

The following fundamental result was shown.
			
\begin{thm}[\cite{V} Theorem 6.1] \label{thm-V}Let $u_0\ge0$ to be in the space $M^{p^*}+L^\infty$. Then there is a time $T>0$ after which the solution $u$ of \eqref{eq-u} becomes  bounded and continuous. 
More precisely, there is a constant $c=c(n)$ such that \begin{equation}T<c\, N^{1-m},
		\end{equation}
where $N=N_{p^*}(u_0):=\lim_{A\to\infty} ||(|f|-A)_+||_{M^{p^*}}$.
	\end{thm}
	\begin{remark}
		It is known that $L^p\subset M^p$ and $N=N_{p^*}(u_0)= 0$ if $u_0\in L^{p^*}+L^{\infty}$  hence in this case $L^\infty$ bound is immediate for $t>0$.   Next, $M^{p^*}+L^\infty\subset M^{p^*}_{loc}$, but they are the same under the decay condition of $u_0$ we assumed. Finally, a typical function $f \in M^{p^*}$, but not in $L^{p^*}$ is $f(x)=\left(\fr{1}{|x|}\right)^{\fr{2}{1-m}}$.  In terms of metric this corresponds to a cylindrical end and the delayed regularity result describes a situation this cylinder shrinks and becomes extinct in a finite time. 
	\end{remark}

We will prove the following result.

	\begin{thm}\label{thm-2}
Assume that   $g=u^{1-m} \, \delta_{ij}$ is  a  solution of the Yamabe flow  \eqref{eq-yamabe} with nonnegative initial data   
$u_0\in  M^{p^*}_{loc}(\R^n)\subset L^1_{loc}(\R^n)$, $p^*=(1-m)\fr{n}{2}$,  such that 
\begin{equation}\label{eqn-initialupperdecay}\limsup_{|x|\to \infty}\left[|x|^2\, u_0^{1-m} -   \fr{(n-2)(n-1)}{\beta} \ln |x|\right]<\infty \end{equation}
for some $\beta >0$. 
Assume in addition that  $u_0$ has a decomposition $u_0=\phi + \psi$ with $\psi \in L^1(\R^n)$ and 
$\phi$ satisfying \eqref{eqn-data}. Then, the rescaled solution $\bar u(x,t) := e^{\fr{2\beta}{1-m} t} u (e^{\beta t}x,t) $ converges 
as $t \to +\infty$,   smoothly on compact sets of $\R^n$, to  $u_{\lambda,\beta}$  which is the unique radial entire solution of 
\eqref{eq-u2} satisfying \eqref{eqn-data}.
\end{thm}

The crucial  step in the proof of  Theorem \ref{thm-2} is to show that   the  upper bound \eqref{eq-t0} holds for some time $t_0$ after delayed regularity. For this,   we will need to prove   that 
the  asymptotics \eqref{eqn-initialupperdecay} of our initial data will  not deteriorate but evolve according to the Yamabe flow.
We achieve this  by  constructing barriers outside of compact balls. 
We will  use the notation $f \sim g \text{ as }r\to \infty$ to indicate that   $\lim_{r\to\infty} f/g =1$.

\begin{prop}[Barrier construction]\label{prop-barrier}  There is $R=R(n,\beta,\lambda)>0$ such that for any $h>R$, the functions 
$$\bar v :=\left(\fr{r^2}{r^2-h^2}\right)^{\fr{1}{1-m}} u_{\beta,\lambda} \qquad \mbox{and} \qquad \underline v := \left(\fr{r^2-h^2}{r^2}\right)^{\fr{1}{1-m}}u_{\beta,\lambda} $$ are a  supersolution and subsolution , respectively,  of the equation \begin{equation}\label{eq-u_t}
u_t=\fr{n-1}{m}\Delta u^m  + \beta x \cdot \nabla  u+ \gamma u, \qquad \text{ on } \big \{ |x|>h \big \} \times (-\infty,\infty). 
\end{equation}	

\end{prop} 

\begin{proof}
In the proof of this proposition, we may fix $\lambda =1$, $\beta=1$ and show that the proposition holds for $R=R(n)$ from the scaling shown in eq \eqref{eq-scaling}. However, we will not use this since it does not makes the proof easy in a significant way.

  We need the following claim. 

\begin{claim}\label{claim-asym}
The solution  radially symmetric solution  $u(x):=u_{\beta,\lambda}(x)=u_{\beta,\lambda}(|x|)$ of \eqref{eq-u2} satisfies 
$$-\fr{n-1}{m}\Delta u^m=\beta \left(ru_r+\fr{2}{1-m}u\right)\sim \fr{\beta}{(1-m)} \fr{u}{\ln r}, \qquad \text{ as }r\to \infty$$ and hence 
$$ru_r\sim -\fr{2}{1-m}u, \qquad  \text{ as }r\to \infty.$$
\end{claim} 
\begin{proof}[Proof of Claim] As in the previous section, we set  $w(s):=r^2u^{1-m}(r)$ and $s=\ln r$. Then,
 $$\begin{aligned}w_s=r\, w_r=r\, (2ru^{1-m}+r^2(1-m)u^{-m}u_r)=r^2 u^{-m}(1-m)\, \left(ru_r+\fr{2}{1-m}u\right).\end{aligned}$$
The claim readily  follows from ${\ds \lim_{r\to \infty}u^{1-m}\fr{r^2}{\ln r}=\fr{(n-1)(n-2)}{\beta}}$ and ${\ds \lim_{s\to\infty}w_s=\fr{(n-1)(n-2)}{\beta}}$. \end{proof}

Denote for simplicity ${\ds f := \left(\fr{r^2}{r^2-h^2}\right)^\fr{1}{1-m}}$ and $u(x):=u_{\beta, \lambda}(x)$ so that  $\bar v = u \, f $. We have 
\begin{equation}\label{eq-comp1}\begin{aligned} \fr{n-1}{m}\Delta\bar v^m +\beta \, (r\bar v_r + \fr{2}{1-m} \bar v )&
= \fr{n-1}{m}( f^m \Delta u^m + u^m\Delta f^m + 2 u^m_r f^m_r ) + \beta \, (r u_r+\fr{2}{1-m}u) f + \beta urf_r \\&= \fr{n-1}{m}(f-f^m)(-\Delta u^m) +\fr{n-1}{m} (u^m\Delta f^m + 2 u^m_r f^m_r ) + \beta  ruf_r\end{aligned}\end{equation}
Meanwhile, \begin{equation}\label{eq-comp2-0}u^m_r f^m_r = m^2 u^{m-1}u_r f^{m-1}f_r\quad\end{equation}
and 
\begin{equation}\label{eq-comp2}\begin{aligned}&\Delta f^m = (mf^{m-1}f_r)_r + \fr{n-1}{r}mf^{m-1}f_r = 
m f^{m-1}f_{rr}+m(n-1) \fr{1}{r}f^{m-1}f_r +m(m-1)f^{m-2}f_r^2.
\end{aligned}
\end{equation}
\begin{equation}\label{eq-comp3}\begin{aligned}
RHS \eqref{eq-comp1}= &\beta urf_r + \fr{n-1}{m}u \, \bigg[ \fr{-\Delta u^m}{u} (f-f^m)+ 2m^2 \fr{ru_r}{u} \fr{u^{m-1}}{r^2} rf^{m-1}f_r\\ 
&+m \fr{u^{m-1}}{r^2} r^2f^{m-1} f_{rr} + m(n-1) \fr{u^{m-1}}{r^2} rf^{m-1}f_r  +  m(m-1)\fr{u^{m-1}}{r^2} r^2 f^{m-2}f_r^2 \bigg].
\end{aligned}\end{equation}
We want to bound all other terms  by first negative term $\beta urf_r$. In that purpose, we compute
$$f = \left(1 + \fr{h^2}{r^2-h^2}\right)^{\fr{1}{1-m}},\qquad f-f^m = (f^{1-m}-1)f^m =\fr{h^2}{r^2-h^2} f^m 
$$
$$rf_r= \fr{r}{1-m}\left(1+\fr{h^2}{r^2-h^2}\right)^{\fr{m}{1-m}} \fr{-2h^2r}{(r^2-h^2)^2}=\fr{f^m}{1-m}\fr{-2h^2r^2}{(r^2-h^2)^2}, $$
$$
r^2f^{m-2}f_r^2 = f^{m-2} \fr{1}{(1-m)^2}f^{2}\fr{4h^4}{(r^2-h^2)^2}=\fr{f^m}{(1-m)^2}\fr{4h^4}{(r^2-h^2)^2}$$ 
$$f^{m-1}rf_r = \fr{1}{1-m} f^m \fr{-2h^2}{(r^2-h^2)} 
$$
and 
\begin{equation*}\begin{aligned}
r^2f^{m-1} f_{rr}& =(r^2-h^2)f_{rr} \\ &=\fr{m(r^2-h^2)}{(1-m)^2}\left( \fr{r^2}{r^2-h^2} \right) ^{\fr{m}{1-m} -1}\fr{4h^4r^2}{(r^2-h^2)^4}+\fr{r^2-h^2}{1-m}\left(\fr{r^2}{r^2-h^2}\right)^\fr{m}{1-m}\left(\fr{2h^2r^2+2h^4}{(r^2-h^2)^3}\right)\\ &=\fr{1}{1-m} f^m \left[ \fr{m}{1-m} \fr{4h^4}{(r^2-h^2)^2} + \fr{ 2(r^2h^2 + h^4)}{(r^2-h^2)^2}\right].
\end{aligned}
\end{equation*}
This shows there is some $C=C(n)>0$, which, in particular, independent of $h$, such that for all $h>1$ on $\{r>h\}$, 
\begin{equation}\label{eqn-all} 
|f-f^m|,\ |r^2 f^{m-2}f_r^2|,\ |r^2f^{m-1}f_{rr}|,\ |rf^{m-1}f_r| \le -C rf_r.
\end{equation}
Using the asymptotics in Claim \ref{claim-asym}, we have
 \begin{equation}\label{eq-compsim}
\fr{-\Delta u^m}{u} \sim \fr{m\beta}{(n-1)(1-m)} \fr{1}{\ln r}, \quad \fr{u^{m-1}}{r^2}\sim \fr{\beta}{(n-1)(n-2)}\fr{1}{\ln r},  \quad
\fr{ru_r}{u}\sim- \fr{2}{1-m}\end{equation}as $r\to \infty$. Combining \eqref{eq-comp2}, \eqref{eqn-all} and  \eqref{eq-compsim}.
shows that there exists  $R=R(\beta, n, \lambda)$ such that for ${\ds \bar v = u \cdot\left(\fr{r^2}{r^2-h^2}\right)^{\fr{1}{1-m}}}$ with $h>R$, 
we have $$\fr{n-1}{m}\Delta \bar v^m  + \beta x \cdot \nabla  \bar v+ \gamma \bar v \le \fr{\beta}{2} \, ruf_r <0\quad   \text{ on }\{r=|x|> h\}.$$ 
This proves that $\bar v$ is a supersolution of \eqref{eq-u_t} in the considered region. 
\smallskip
For the $\underline v = u \left(\fr{r^2-h^2}{r^2}\right)^\fr{1}{1-m} = u \, g$, equations \eqref{eq-comp1}, \eqref{eq-comp2-0}, \eqref{eq-comp2}, and \eqref{eq-comp3} are the same except $\bar v$ and $f$ changed by $\underline v$ and $g$. We compute,
$$
g = \left(1 - \fr{h^2}{r^2}\right)^{\fr{1}{1-m}}, \qquad g-g^m = (g^{1-m}-1)g^m =-\fr{h^2}{r^2} g^m
$$
$$ rg_r =r \fr{1}{1-m}\left(1-\fr{h^2}{r^2}\right)^{\fr{m}{1-m}} \fr{2h^2}{r^3}=\fr{1}{1-m}g^m\fr{2h^2}{r^2}, \quad g^{m-1}rg_r = \fr{1}{1-m} g^m \fr{2h^2}{r^2-h^2}$$
$$r^2g^{m-2}g_r^2 = \fr{1}{(1-m)^2}g^{m-2} g^{2m}\fr{4h^4}{r^4}=4 \fr{g^m}{(1-m)^2} \left(\fr{h^2}{r^2-h^2}\right)^2$$
and 
\begin{equation}
\begin{split}
r^2g^{m-1} g_{rr} &=\fr{r^4}{r^2-h^2}g_{rr} \\
&=\fr{r^4}{r^2-h^2}\fr{m}{(1-m)^2}\left(\fr{r^2-h^2}{r^2} \right) ^{\fr{m}{1-m} -1}\fr{4h^4}{r^6}+\fr{r^4}{r^2-h^2}\fr{1}{1-m}\left(\fr{r^2-h^2}{r^2}\right)^\fr{m}{1-m}\left(\fr{-6h^2}{r^4}\right)\\
&=\fr{1}{1-m} g^m \left[ \fr{m}{1-m} 4\left(\fr{h^2}{r^2-h^2}\right)^2 - \fr{6h^2}{r^2-h^2}\right].
\end{split} 
\end{equation}

Since ${\ds \left(\fr{h^2}{r^2-h^2}\right)^2}$ in $\ds r^2 g^{m-2}g_r^2 $ dominates all other terms appearing above, namely ${\ds \left(\fr{h^2}{r^2-h^2}\right),\ \fr{h^2}{r^2}}$ and ${\ds  \fr{h^4}{r^6}}$ near $r=h$, we may combine \eqref{eq-comp3} 
and \eqref{eq-compsim} to find $R_1=R_1(n,\beta, \lambda)>0$ and $\delta(n,\beta, \lambda)>0$ such that for $h>R_1$ and $h<r\le(1+\delta)h$, $$\fr{n-1}{m}\Delta \underline v^m  + \beta x \cdot \nabla  \underline v+ \gamma \underline v \ge \fr{n-1}{2}u^mg^{m-1}g_{rr}>0
\qquad \mbox{on} \,\, h<r\le(1+\delta)h.$$
On the remaining region $r>(1+\delta)h$, there is $C=C(n,\delta)>0$ such that \begin{equation}
|g-g^m|,\ |r^2 g^{m-2}g_r^2|,\ |r^2g^{m-1}g_{rr}|,\ |rg^{m-1}g_r| \le C\,  rg_r.
\end{equation}
Combining again \eqref{eq-comp3} and  \eqref{eq-compsim}, it follows that for each $\delta>0$ there is $R_2=R_2(\beta, n, \lambda, \delta)$ such that for ${\ds \underline v = u\cdot\left(\fr{r^2-h^2}{r^2}\right)}$ with $h>R_2$,
$$\fr{n-1}{m}\Delta \underline v^m  + \beta x \cdot \nabla  \underline v+ \gamma \underline v \ge \fr{\beta}{2} urg_r>0 \quad \text{ on }
\{r> (1+\delta)h\}.$$ 
Setting  $R:=\max(R_1,R_2)$, it follows that $\underline v$ is a subsolution on the region   $\{r> h\}$, for $h > R$,
concluding the proof of the proposition. 
\end{proof}

Using the previous barrier construction we will now show that the Yamabe flow preserves the asymptotic behavior of
our initial data $u_0$ as in Theorem \ref{thm-2}. 

\begin{prop}\label{prop-comp}
Let    $u_0\in L^1_{loc }(\R^n)$ satisfying  ${\ds \limsup_{r\to \infty}\left[r^2u_0^{1-m} -   \fr{(n-2)(n-1)}{\beta} \ln r\right]=K_1<\infty}.$
Then, 	the solution $u$ of \eqref{eq-flatyamabe}  with initial data $u_0$ satisfies 
 $$\limsup_{r\to \infty}\left[r^2u^{1-m} -   \fr{(n-2)(n-1)}{\beta}\right] \le K_1-{(n-1)(n-2)}t, \quad  \text{ for } t\ge0.$$ 
Also if ${\ds \liminf_{r\to \infty}\left[r^2u_0^{1-m} -   \fr{(n-2)(n-1)}{\beta} \ln r\right]=K_2 }>-\infty$, then 
	 $$K_2-{(n-1)(n-2)}t\le \liminf_{r\to \infty}\left[r^2u^{1-m} -   \fr{(n-2)(n-1)}{\beta}\right], \quad   \mbox{ for } t\ge0.$$
	 \begin{proof}
By Proposition \ref{prop-asymp},   there exists  $\lambda_1>0$ such that $$\lim_{r\to \infty}\left[r^2u_{\beta,\lambda_1}^{1-m} -   \fr{(n-2)(n-1)}{\beta} \ln r\right]= K_1.$$ For   each $\epsilon>0$ there exists  
$h>R(n,\beta,\lambda_1)$,  (where $R(n,\beta,\lambda_1)$ is taken by Proposition \ref{prop-barrier}) such that $$\text{$u_0\le u_{\beta,\lambda_1+\epsilon}\left(\fr{r^2}{r^2-h^2}\right)^{\fr{1}{1-m}}:=\bar v_{\beta,\lambda_1+\epsilon}$, \quad  on $\{r>h\}$}.$$
		 Since ${\ds \bar v_{\beta,\lambda_1+\epsilon}:= u_{\beta,\lambda_1+\epsilon}\left(\fr{r^2}{r^2-h^2}\right)^{\fr{1}{1-m}} \to \infty}$ as $r\to h+$,  the comparison gives us that  $\bar u(x,t):= e^{ \gamma t}u(e^{\beta t}x,t)\le \bar v_{\beta,\lambda_1+\epsilon}$ on $r>h$ and $t>0$. 
		 Also, since ${\ds \left (  \frac {r^2}{r^2-h^2} \right)^{\fr{1}{1-m}}  \to 1}$, as $r \to +\infty$, using  \eqref{eqn-kappa} we conclude
		 that  $$\text{$\lim_{r\to \infty}\left[r^2\bar v_{\beta,\lambda_1+\epsilon}^{1-m} -   \fr{(n-2)(n-1)}{\beta} \ln r\right]= K_1+\fr{2(n-1)(n-2)}{(n+2)\beta}\ln(1+\fr{\epsilon}{\lambda_1})$}.$$ This translates into the following asymptotics of $u(x,t)$ whihc holds for  for each $t>0$ 
		 	 $$\limsup_{r\to \infty}\left[r^2u^{1-m} -   \fr{(n-2)(n-1)}{\beta} \, \ln r\right] \le K_1+\fr{2(n-1)(n-2)}{(n+2)\beta}\ln\left(1+\fr{\epsilon}{\lambda_1}\right)-{(n-1)(n-2)}t.$$ Taking the limit  $\epsilon\to 0+$ we reach our conclusion. 
			 The other side inequality can  be done similarly by comparison with the constructed subsolution 
			 ${\ds {\underline v}_{\ \beta,\lambda_1-\epsilon}:= u_{\beta,\lambda_1-\epsilon}\left(\fr{r^2-h^2}{r^2}\right)^{\fr{1}{1-m}}}$. 
	 	\end{proof}	\end{prop}
	 	
We will mow conclude the proof of Theorem \ref{thm-2}. 		
	
\begin{proof}[Proof of Theorem \ref{thm-2}]
We have seen in  Theorem \ref{thm-1} that the rescaled solution $\bar u(x,t):=e^{\fr{2\beta}{1-m}t} u(e^{\beta t}\, x)$ converges
in $L^1_{loc}$, as $t \to +\infty$, to the steady soliton $ u_{\beta,\lambda}$.  
By the the discussion following Theorem \ref{thm-1}, to establish  the   $C_{loc}^\infty$ convergence,  it suffices to show that $\bar u(t)$ satisfies a uniform in time  $L^\infty_{loc}$ bound for $t \geq T$, for some $T >0$.   

Indeed, by  Theorem \ref{thm-V}, there is a finite time $T>0$ such that $||u(t)||_{L^\infty(\R^n)}<\infty$ for $t\ge T$. At $t=T$, Proposition \ref{prop-comp} implies$$\limsup_{|x|\to \infty}\left[|x|^2u(T)^{1-m} -   \fr{(n-2)(n-1)}{\beta} \ln |x|\right]<\infty .$$
We may now combine the  $L^\infty$ bound on $u(T)$ and this asymptotic behavior (similarly as in the proof of Corollary \ref{cor-uniform})   to show that there exists 
$\lambda_2 >0$ for which $$u(x,T)\le e^{-\fr{2\beta}{1-m}T}u_{\beta,\lambda_2}(e^{-\beta T}x).$$
This implies the bound $\bar u (x,t) \le  u_{\beta,\lambda_2}(x) $ for $t\ge T$, from which the $L^\infty_{loc}$ bound on $\bar u$
readily follows. This concludes the proof of the theorem. 
\end{proof}

\section{Examples of Type II Singularity }\label{sec-type2}
In this last section, we will construct noncompact conformally flat solutions $g=u\, \delta_{ij}$ of the Yamabe flow
\eqref{eq-yamabe} which admit  {\em   type II singularities}   both in a {\em finite time} and {\em infinite time}.  Before we start, let us fix
the following notation. 
\begin{notation}For any fixed $\beta>0$ and $\lambda>0$, we  denote by
\begin{itemize}
\item $u_{\beta,\lambda}$ (gradient Yamabe steady soliton) to be the unique radial solution of equation \eqref{eq-u} 
with $\beta$, $\gamma=\fr{2\beta}{1-m}$ and $u_{\beta,\lambda}(0)=
\lambda\text{, and}$

\item $v_{\beta,\lambda}$  (gradient Yamabe shrinker soliton) to be the unique radial solution with $\beta$, $\gamma=\fr{2\beta+1}{1-m}$ and $v_{\beta,\lambda}(0)=\lambda.$
\end{itemize}
\end{notation}

\begin{definition}
	Suppose that a solution $g(t)$ to Yamabe flow \eqref{eq-yamabe} on $t\in[0,T)$ has a singularity at $t=T<\infty$;  this finite time singularity is called type I if  $$ \sup_{M\times [0,T)} |Rm(x,t)|(T-t)<\infty,$$  and is called type II if $$ \sup_{M\times [0,T)} |Rm(x,t)|(T-t)=\infty.$$
\end{definition}

\begin{definition}
	A solution to Yamabe flow \eqref{eq-yamabe} on $t\in[0,\infty)$ is called type I if  $$ \sup_{M\times [0,\infty)} |Rm(x,t)|<\infty,$$  and is called type II if $$ \sup_{M\times [0,\infty)} |Rm(x,t)|=\infty.$$
\end{definition}

Before we proceed, we begin with the next simple observation. 

\begin{lemma}\label{lemma-decay} Let $g(t)=u^{1-m} \, \delta_{ij} $ on $t\in[0,T)$ be a  solution of the Yamabe flow \eqref{eq-yamabe} 
such that the scalar curvature satisfies  $R(x,t)\le f(t)\in L^1_{loc}([0,T))$. Then the conformal factor $u$ satisfies a pointwise estimate $$\text{$u(x,t)^{1-m}\ge 
u_0(x)^{1-m}e^{-\int_0^t f(s)ds}$ for $t\in[0,T)$.}$$ In particular, $R(x,t)\le \fr{K}{T-t}$ implies that $u(x,t)\ge u_0(x) \, (T-t)^{\fr{K}{1-m}}$ and $R(x,t)\le K$  that $u(x,t)\ge u_0(x) \,  e^{-\fr{K}{1-m}   t}.$
	\begin{proof}This is a straightforward ODE estimate. 
		For each fixed $x\in M$, the function $\phi(x,t):=u^{1-m}(x,t)$ satisfies $\phi_t=-R \, \phi$,  hence $(\log \phi)_t=-R\ge-f$. Integrating in time  gives the result.
		\end{proof}
\end{lemma}

\begin{theorem}\label{thm-typeII-1} Suppose $0<u_0\in L^\infty (\R^n) \cap C(\R^n)$  and satisfies the bound  $|x|^2u_0^{\fr{4}{n+2}} < {(n-1)(n-2)T}$,
  for some $T>0$, and the asymptotic behavior 
$$\lim_{|x|\to\infty} \left||x|^2 \, u_0^{\fr{4}{n+2}} - {(n-1)(n-2)T}\right|=0 \quad \mbox{and} \quad  
\lim_{|x|\to\infty} |x|^\epsilon\left ||x|^2 \, u_0^{\fr{4}{n+2}} - {(n-1)(n-2)T}\right| =\infty, \forall \epsilon >0.$$
	 Then, the solution $g=u^{1-m}\, \delta_{ij}$ of the Yamabe flow \eqref{eq-yamabe} with initial data $g_0=u_0^{1-m} \delta_{ij}$
	 becomes extinct at time $T$ and develops a  type II singularity at $t=T$.
	
	\begin{proof}
		The fact that the unique smooth solution $u$ of \eqref{eq-u} with initial data $u_0$  exists at least up to time $t=T$,  can be easily seen   by comparing  with family of 
		Barrenblatt   solutions which extinct at $T-\epsilon$, as  in Lemma 4.2 in \cite{DKS}. On the other hand because
		of our initial bound  $|x|^2u_0^{\fr{4}{n+2}} < {(n-1)(n-2)T}$, by comparing  with the shrinking cylinder which vanishes at $t=T$, we know that our solution becomes  extinct at $t=T$. 
		
		It suffices to prove that the singularity is of type II. We argue by contradiction and suppose that it is of type I,
		which means that   there is $K>0$ such that  
		$$\fr{|R|}{T-t}\le C_n\fr{|Rm|}{T-t}\le C_n\, K.$$ Let us fix  $\beta>0$ so that $$\gamma=\fr{2\beta+1}{1-m}>\fr{C_nK}{1-m}.$$  Then, for this choice of $\beta$, there is an one parameter family $\{ v_\lambda \}_{\lambda >0}$ of  radial solutions (shrinkers) of \eqref{eq-u}, with $v_\lambda(0)=\lambda$. 
	
\begin{claim}\label{claim-cl1} There exist   a large $\lambda_0 >0$   such that $T^\gamma v_{\lambda_0} ( xT^\beta )\ge u_0(x) $,   for all $x\in\R^n$.\end{claim}
	\begin{proof}[Proof of Claim \ref{claim-cl1}]		By a scaling argument, we can assume $T=1$.
		First, choose any $\lambda_1$ with $v_{\lambda_1}(0)=\lambda_1>u_0(0)$. If the claim holds for this $\lambda_1$, we are done. If not, we first recall   asymptotics at infinity for  the  conformally flat shrinker $g_\lambda:=v_\lambda\, \delta_{ij}$ shown in \cite{DKS}, namely \begin{equation}\label{eq-asym}|x|^2 \, v_{\lambda_1}^{1-m}(x)=(n-1)(n-2)-B \, |x|^{-\gamma}+o(|x|^{-\gamma})\end{equation} 
		as $|x|\to\infty$ for some $0<B=B(n,\beta,\lambda_1)$ and $0<\gamma=\gamma(n,\beta)$. Our assumed conditions on the initial data $u_0$ 
and 		 \eqref{eq-asym}  imply  that ${\mathcal K} :=\{x\in\R^n \, | \, u_0\ge v_{\lambda_1}\}$ is a compact set which doesn't contain the origin
(since $v_{\lambda_1}(0)=\lambda_1>u_0(0)$). Next, we can observe that 
		$$|x|^2 \, v_\lambda(x)^{1-m}=((\lambda/\lambda_1)^{\fr{1-m}{2}}|x|)^2 v_{\lambda_1}((\lambda/\lambda_1)^{\fr{1-m}{2}}x) .$$ 
		This and \eqref{eq-asym} imply $|x|^2 \, v_\lambda(x)^{1-m} \to (n-1)(n-2)$, as $\lambda \to \infty$,  uniformly on ${\mathcal K}$ while $|x|^2u_0(x)<(n-1)(n-2)$ on ${\mathcal K}$. Using this uniform convergence, therefore, we may find some $\lambda_0>\lambda_1$ such that $|x|^2v_{\lambda_0}(x)> |x|^2 u_0(x)$ on ${\mathcal K}$, namely $v_{\lambda_0}(x)>  u_0(x)$ on ${\mathcal K}$.  
On the other hand, the monotonicity of $v_\lambda$ with respect to $\lambda$ implies that  $v_{\lambda_0}> v_{\lambda_1} >u_0$ on 
$\R^n\setminus \mathcal K$, concluding that $v_{\lambda_0} > u_0$ on $\R^n$. 
	\end{proof}

We will now conclude the proof of the theorem. By the comparison principle,   $(T-t)^\gamma v_{\lambda_0}( x(T-t)^\beta)\ge u(x,t)$,  on $t\in[0,T)$. On the other hand, by Lemma \ref{lemma-decay}, ${\ds u(x,t)\ge u_0(x)(T-t)^{\fr{C_nK}{1-m}}}$. In particular, at $x=0$, we have $$(T-t)^\gamma v_{\lambda_0}(0)\ge u(0,t) \ge u_0(0)(T-t)^{\fr{C_nK}{1-m}}.$$ Since $\gamma> \fr{C_nK}{1-m}>0$ and $u_0(0) >0$, there must be some $t<T$ close to $T$ so that above inequality fails to hold, leading to a contradiction.
We conclude that  the singularity must be of type II.
		
			\end{proof}
	\end{theorem}

\begin{theorem}\label{thm-typeII-2}  Suppose $0<u_0\in L^\infty (\R^n) \cap C(\R^n)$  and satisfies  
$$\lim_{|x|\to\infty} |x|^2u_0^{\fr{4}{n+2}}=\infty \qquad   \text{ and} \qquad \lim_{|x|\to\infty} \fr{|x|^2}{\ln|x|}u_0^{\fr{4}{n+2}} =0.$$
	Then, the solution $g=u^{1-m}\, \delta_{ij}$ of the Yamabe flow \eqref{eq-yamabe} with initial data $g_0=u_0^{1-m}\, \delta_{ij}$ exist globally on $0 \leq t < +\infty$  and develops a  type II singularity as $t \to \infty$.
	\begin{proof}
		The proof  is very  similar to that of Theorem \ref{thm-typeII-1} where the shrinkers $v_\lambda$ are replaced by the
		steady solito $u_\lambda$. The global in time existence with such initial condition is well known, for instance in Theorem 1.1 in \cite{H2} , hence it suffices to prove that the solution develops a type II singularity at $t=\infty$. 
		Suppose it is type I and so that there is $K>0$ such that  ${|Rm|}\le K$ and hence $${|R|}\le C_n\, |Rm|\le C_n \, K.$$ Let us choose any $\beta>0$ such that $\gamma=\fr{2\beta}{1-m}>\fr{C_nK}{1-m}$. For this fixed $\beta>0$,  there is an one parameter family 
		$\{ u_\lambda \}_{\lambda >0}$ of 
		radial solutions  of \eqref{eq-u}  with $u_\lambda(0)=\lambda$. 
		
		Using the asymptotics of a steady soliton and the observation that  $ \inf_{\mathcal K} u_\lambda \to \infty$,  as $\lambda\to\infty$,
		 for each compact ${\mathcal K} \subset \R^n$ (Corollary \ref{cor-uniform}), we may find large $\lambda_0 >0$ such that $ u_{\lambda_0} ( x )> u_0(x) $  for all $x\in\R^n$. Thus $e^{-\gamma t} u_{\lambda_0}( xe^{-\beta t })\ge u(x,t)$ on $t\in[0,\infty)$,  by the  comparison principle. On the other hand, by Lemma \ref{lemma-decay}, $u(x,t)\ge u_0(x)e^{-\fr{C_nK}{1-m}t}$. In particular, at $x=0$, we have $$e^{-\gamma t}v_{\lambda_0}(0)\ge u(0,t) \ge u_0(x)e^{-\fr{C_nK}{1-m}t}.$$ Since $\gamma> \fr{C_nK}{1-m}>0$, there must be some $t$ large so that above inequality fails, leading to a contradiction. We conclude that  the singularity of the solution $u$ must be  type II.
		\end{proof}
\end{theorem}

\centerline{\bf Acknowledgements}

\smallskip 

\noindent  P. Daskalopoulos  has been partially supported by NSF grants  DMS-1266172 and  DMS-1600658.  Beomjun Choi has been partially supported by NSF grants  DMS-1600658.


\end{document}